\documentclass[12pt]{amsart}
\usepackage{fullpage, amsfonts,amssymb,amsmath,verbatim}
\usepackage[alphabetic,backrefs,lite]{amsrefs} % for bibliography
\usepackage{setspace}
\usepackage{amsthm}
\usepackage[all]{xy}

\newtheorem{lemma}{Lemma}[section]

\newtheorem{cor}[lemma]{Corollary}
\newtheorem{conj}[lemma]{Conjecture}

\newtheorem{claim*}{Claim}
\newtheorem{thm}[lemma]{Theorem}

\newtheorem{example}[lemma]{Example}

\theoremstyle{remark}
\newtheorem{remark}[lemma]{Remark}

\newcommand{\reg}{\operatorname{reg}}

\newcommand{\rank}{\operatorname{rank}}
\newcommand{\codim}{\operatorname{codim}}

\newcommand{\defi}[1]{{\upshape\sffamily #1}}

\title[A special case of the BEH rank conjecture]{A special case of the Buchsbaum-Eisenbud-Horrocks rank conjecture}
\author{Daniel Erman}
\address{Department of Mathematics, University of California,
        Berkeley, CA 94720-3840, USA}
\email{derman@math.berkeley.edu}
\urladdr{http://math.berkeley.edu/\~{}derman}
\thanks{
The author is partially supported by an NDSEG fellowship.
}
\date{\today}

\begin{document}
\maketitle
\begin{abstract}
The Buchsbaum-Eisenbud-Horrocks rank conjecture proposes lower bounds for the Betti numbers of a graded module $M$ based on the codimension of $M$.  We prove a special case of this conjecture via Boij-S\"oderberg theory.  More specifically, we show that the conjecture holds for graded modules where the regularity of $M$ is small relative to the minimal degree of a first syzygy of $M$.  Our approach also yields an asymptotic lower bound for the Betti numbers of powers of an ideal generated in a single degree.
\end{abstract}

\section{Introduction}
Let $k$ be any field, let $S=k[x_1, \dots, x_n]$ be the polynomial ring with the usual grading, and let $M$ be a graded $S$-module.  The Buchsbaum-Eisenbud-Horrocks rank conjecture (herein the BEH rank conjecture) says roughly that the Koszul complex is the ``smallest'' possible minimal free resolution.\footnote{Terminology for this conjecture is inconsistent in the literature.  In some places this conjecture is known as Horrocks' Conjecture or as the Syzygy Conjecture.}   The conjecture was formulated by Buchsbaum and Eisenbud in \cite[p.\ 453]{buchs-eis-gor} and, independently, the conjecture is implicit in a question of Horrocks \cite[Problem 24]{hartshorne-vector}.  Although the conjecture is most commonly phrased for regular local rings, we consider the graded case.  Let
$$
\xymatrix{
0\ar[r]& F_p \ar[r]^{\phi_p}& F_{p-1}\ar[r]^{\phi_{p-1}}&\dots \ar[r]^{\phi_1}& F_0 \ar[r]^{\phi_0}& M\ar[r]& 0
}
$$
be the graded minimal free resolution of $M$.  Let $\beta_j(M):=\rank(F_j)$.

\begin{conj}[Graded BEH Rank Conjecture]\label{conj:beh}
Let $M$ be a graded Cohen-Macaulay $S$-module of codimension $c$.  Then:
\[
\beta_j(M)\geq \binom{c}{j}
\]
for $j=0, \dots, c$.
\end{conj}

In this paper, we prove a special case of the graded BEH rank conjecture.  We do {\em not} require that $M$ is Cohen-Macaulay.
\begin{thm}\label{thm:weak}
Let $M$ be a graded $S$-module of codimension $c$, generated in degree $\leq 0$, and let $\underline{d}_1(M)$ be the minimal degree of a first syzygy of $M$.  If $\reg(M) \leq 2\underline{d}_1(M)-2$, then\[
\beta_j(M)\geq \beta_0(M)\binom{c}{j}.
\]
for $j=0, \dots, c$.
\end{thm}

\begin{figure}\label{fig:shape}
	\begin{equation*}
		\bordermatrix{
		& 0& 1& 2 & \dots & p \cr
		0& *& - & - & \dots & - \cr
		1& -& - & - & \dots & - \cr
		& & & & & \cr
		\vdots& \vdots& \vdots  & \vdots &\vdots &\vdots \cr
		& & & & & \cr
		\underline{d}_1-2&- & - & - & \dots & - \cr
		\underline{d}_1-1&- & * & * & \dots & * \cr
		& & & & & \cr
		\vdots& \vdots& \vdots  & \vdots &\vdots &\vdots \cr
		& & & & & \cr
		2\underline{d}_1-3& - & * &*& \dots & * \cr
		2\underline{d}_1-2& -& * & * & \dots & *
		}
	\end{equation*}
	\caption{When $M$ has a Betti diagram of the above shape, then it satisfies the Buchsbaum-Eisenbud-Horrocks Rank Conjecture.}
\end{figure}

Common generalizations of the BEH rank conjecture include removing the Cohen-Macaulay hypothesis and/or strengthening the conclusion to the statement that $\rank(\phi_j)\geq \binom{c-1}{j-1}$ for $j=1, \dots, c-1$.  A different generalization, suggested in \cite[Conj II.8]{carlsson}, replaces the Betti numbers of a free resolution by the homology ranks of a differential graded module.

The BEH rank conjecture has been shown to hold for all modules of codimension at most $4$ \cite[p.\ 267]{evans-griffith}.  In codimension at least $5$, however, the BEH rank conjecture has only been settled for families of modules with additional structure.

Theorem~\ref{thm:weak} applies to modules whose Castelnuovo-Mumford regularity is small relative to the degree of the first syzygies of $M$.  Though the literature on special cases of the BEH rank conjecture is extensive, Theorem~\ref{thm:weak} moves in a new direction.  The most similar result in the literature is perhaps \cite[Thm.\ 0.1]{chang}, which shows that the BEH rank conjecture holds when $M$ is a Cohen-Macaulay module annihilated by the square of the maximal ideal $\mathfrak m$.  Other known cases of the BEH rank conjecture include multigraded modules~\cite[Thm.\ 3]{chara-multi} and \cite{santoni}, cyclic modules in the linkage class of a complete intersection~\cite{huneke-ulrich}, cyclic quotients by monomial ideals~\cite[Cor 2.5]{evans-griffith}, and several more \cite{chang2}, \cite{charalamb}, \cite{dugger}, and \cite{hochster-richert}.  See \cite[pp.\ 25-27]{chara-evans-problems} for an expository account of some of this progress.

The method of proof for Theorem~\ref{thm:weak} is quite different from previous work on the BEH rank conjecture.  Our proof is an application of Boij-S\"oderberg theory, by which we mean the results of \cite{eis-schrey} and \cite{boij-sod2}.  At first glance, it might appear that Boij-S\"oderberg theory would not apply to Conjecture~\ref{conj:beh}: Boij-S\"oderberg theory is based on the principle of only considering Betti diagrams up to scalar multiple, whereas the BEH rank conjecture depends on the integral structure of Betti diagrams.  However, if the Betti diagram of $M$ has shape as in Figure~\ref{fig:shape}, then this imposes conditions on the pure diagrams which can appear in the Boij-S\"oderberg decomposition of $M$.  This allows us to reduce the proof of Theorem~\ref{thm:weak} to a statement about the numerics of pure diagrams.  We then use a multivariable calculus argument to degenerate the relevant pure diagrams to a Koszul complex.

Our analysis of the numerics of pure diagrams also leads to a proposition about the asymptotic behavior of the Betti numbers of $S/I^t$ where $I$ is an ideal generated in a single a degree $\delta$.  Let $c=\codim(S/I)$ and let $b$ be the asymptotic regularity defect of $I$ (see \S\ref{sec:asymp} for a definition).  We will show that
\[
\beta_j(S/I^t) \geq \frac{(b!)^2\delta^{c-1}}{(j-1+b)!(c-j+b)!}t^{c-1} + O(t^{c-2})
 \]
for all $t\gg0$ and all $j\in \{1, \dots, c\}$.

This paper is organized as follows.  In \S\ref{sec:boijsod} we review the relevant aspects of Boij-S\"oderberg theory.  In \S\ref{sec:pure}, we investigate the numerics of pure diagrams which satisfy the conditions of Theorem~\ref{thm:weak}.  This analysis of pure diagrams is the foundation for the proof of Theorem~\ref{thm:weak}, which appears in \S\ref{sec:proof}.  In \S\ref{sec:asymp}, we prove Proposition~\ref{thm:asymp} about asymptotic Betti numbers.  In \S\ref{sec:apps}, we consider applications of Theorem~\ref{thm:weak}.

%%%%%%%%%%
%%%%%%%%%%
\section{Review of Boij-S\"oderberg theory}\label{sec:boijsod}
%%%%%%%%%%
%%%%%%%%%%
We say that a sequence $\mathbf{d}=(d_0, \dots, d_s)\in \mathbb Z^{s+1}$ is a \defi{degree sequence} if $d_i>d_{i-1}$ for all $i>0$.  We use the notation $\deg(\mathbb Z^{s+1})$ for the space of degree sequences  of length $s+1$.  Given two degree sequence $\mathbf{d},\mathbf{d}'\in \mathbb Z^{s+1}$, we say that $\mathbf{d}\geq \mathbf{d}'$ if $d_i\geq d_i'$ for $i=0, \dots, s$.  If $s\leq p$ and $\mathbf{d}=(d_0, \dots, d_p)$ is a degree sequence, then we define $\tau_s(\mathbf{d})$ to be the truncated degree sequence $\tau_s(\mathbf{d}):=(d_0, \dots, d_s)$.

Each degree sequence $\mathbf{d}$ defines a ray in the cone of Betti diagrams; there exists a unique point  $\overline{\pi}(\mathbf{d})$ on this ray with $\beta_{0,d_0}(\overline{\pi}(\mathbf{d}))=1$ (c.f. \cite[Thm 0.1]{eis-schrey}).  The diagram $\overline{\pi}(\mathbf{d})$ is the \defi{normalized pure diagram of type} $\mathbf{d}$.

Note that $\overline{\pi}(\mathbf{d})$ may have non-integral entries.  For instance
\[
\overline{\pi}(0,1,2,4)=
\begin{pmatrix}
1&\frac{8}{3}&2&-\\
-&-&-&\frac{1}{3}
\end{pmatrix}.
\]
Given any diagram $D$ we define $\beta_{i}(D):=\sum_j \beta_{i,j}(D)$.

Let $M$ be a graded $S$-module of codimension $c$ and projective dimension $p$.  For $i=0, \dots, p$ we define $\underline{d}_i:=\min \{j | \beta_{i,j}(M)\ne 0\}$ and $\overline{d}_i:=\max \{j | \beta_{i,j}(M)\ne 0\}$. We set $\underline{\mathbf{d}}:=(\underline{d}_0, \dots, \underline{d}_p)$ and $\overline{\mathbf{d}}:=(\overline{d}_0, \dots, \overline{d}_p)$.  Boij-S\"oderberg theory shows that the Betti diagram of any graded $S$-module can be expressed as a positive rational sum of pure diagrams which correspond to degree sequences bounded by $\overline{\mathbf{d}}$ and $\underline{\mathbf{d}}$.  The following theorem is weaker than the main results of Boij-S\"oderberg theory, but it will be sufficient for our purposes.

\begin{thm}[Eisenbud-Schreyer (2008), Boij-S\"oderberg (2008)]\label{thm:boijsod}
Let $M$ be a graded $S$-module of projective dimension $p$ and codimension $c$ and let $\overline{\mathbf{d}}$ and $\underline{\mathbf{d}}$ as above.
%\begin{enumerate}
%	\item ({\bf Cohen-Macaulay Case}):  If $M$ is Cohen-Macaulay, then the Betti diagram $\beta(M)$ can be expressed as a sum:
%	$$\beta(M)=\sum_{\overset{\mathbf{d}\in \deg(\mathbb Z^{p+1})}{\underline{\mathbf{d}}\leq \mathbf{d} \leq \overline{\mathbf{d}}}}a_{\mathbf{d}} \overline{\pi}(\mathbf{d})$$
%where each $a_{\mathbf{d}}$ is a nonnegative rational number.
%	\item ({\bf General Case}):
The Betti diagram $\beta(M)$ can be expressed as a sum:
	$$\beta(M)=\sum_{c\leq s\leq p}\sum_{\overset{\mathbf{d}\in \deg(\mathbb Z^{s+1})}{\tau_s(\underline{\mathbf{d}})\leq \mathbf{d} \leq\tau_s( \overline{\mathbf{d}})}}a_{\mathbf{d}} \overline{\pi}(\mathbf{d}) $$
	where each $a_{\mathbf{d}}$ is a nonnegative rational number.
%\end{enumerate}
\end{thm}
A stronger version of this theorem allows for a unique decomposition of $\beta(M)$ into a positive rational sum of pure diagrams and even provides an algorithm for producing this decomposition.  See the introduction of \cite{eis-schrey} for an expository treatment of the main results of Boij-S\"oderberg theory, and see \cite[\S7]{eis-schrey} for a proof of of Theorem~\ref{thm:boijsod} in the case where $M$ is Cohen-Macaulay.  See \cite[Thm 2]{boij-sod2} for a proof of Theorem~\ref{thm:boijsod} when $M$ is not necessarily Cohen-Macaulay.

%%%%%%%%%%
%%%%%%%%%%
\section{Ranks of Pure Diagrams}\label{sec:pure}
%%%%%%%%%%
%%%%%%%%%%

In this section, we study the numerics of the pure diagrams which satisfy the conditions of Theorem~\ref{thm:weak}.  Given any degree sequence $\mathbf{d}$, the Betti numbers of $\overline{\pi}(\mathbf{d})$ can be expressed as a rational function in the $d_i$, and we will use these rational functions to investigate the pure diagrams.

We introduce auxiliary functions to simplify the notation.  For $\mathbf{e}=(e_1, \dots, e_s)\in \mathbb R^s$, we define linear functions, $T_{i}, U_{i,j}, $ and $V_{i,j}$ from $\mathbb R^s$ to $\mathbb R$ by the following formulas:
\begin{align*}
T_{i}(\mathbf{e})&:=i+e_1+e_2+\dots +e_i, \text{ for } i=1, \dots, s\\
U_{i,j}(\mathbf{e})&:=(j-i+1)+e_{i}+e_{i+1}+\dots e_{j}, \text{ whenever } i<j\\
V_{i,j}(\mathbf{e})&:=(i-j)+e_{j+1}+\dots +e_{i}, \text{ whenever } i>j
\end{align*}

Let $\mathfrak d: \mathbb R^s \to \mathbb R^{s+1}$ be the linear map:
\[
\mathfrak d_j(\mathbf{e})=
\begin{cases} 0 & \text{ for } j=0\\
j+\sum_{i=0}^j e_i, & \text{ for } j=1, \dots, s.
\end{cases}
\]
Note that $\mathbf{e}\in \mathbb Z^s_{\geq 0}$ if and only if $\mathfrak d(\mathbf{e})$ is a degree sequence with first entry equal to $0$.

We define the rational function $\mathfrak b_j: \mathbb R^s\to \mathbb R$ by:
\begin{equation}\label{eqn:betaj}
\mathfrak b_j(\mathbf{e}):=\frac{\left(\prod_{i=1, i\ne j}^s T_i\right)} {\left(\prod_{i=2}^{j-1} U_{i,j} \right) \left(\prod_{i=j+1}^s V_{i,j}\right)}
\end{equation}
for $j=1, \dots, s$.  The rational function $\mathfrak b_j$ has no poles on $\mathbb R^s_{\geq 0}$.  The purpose of these definitions is summarized in the following lemma:
\begin{lemma}\label{lem:defs}
Let $\mathbf{e}\in \mathbb Z^s_{\geq 0}$.  Then we have:
\[
\mathfrak b_j(\mathbf{e})=\beta_j\left(\overline{\pi}\left(\mathfrak d\left(\mathbf{e}\right)\right)\right)
\]
\end{lemma}

\begin{proof}
For any degree sequence $\mathbf{d}$ of length $s$, a result of \cite{herz-kuhl} can be used to give the explicit formulas
\[
\beta_j\left(\overline{\pi}\left(\mathbf{d}\right)\right)=\prod_{\overset{1\leq i \leq s}{i\ne j}} \frac{|d_i-d_0|}{|d_i-d_j|}
\]
(c.f. \cite[Defn 2.3]{boij-sod}).  Now let $\mathbf{d}=\mathfrak d\left(\mathbf{e}\right)$ and fix some $i\ne j$.  Observe that $|d_i-d_0|=T_i$; further, $|d_i-d_j|=U_{i,j}$ if $i<j$ and $|d_i-d_j|=V_{i,j}$ if $i>j$.  This proves the lemma.
\end{proof}

\begin{lemma}\label{lem:upshifting}
On the domain $\mathbf{e}\in \mathbb R^s_{\geq 0}$, we have $\frac{\partial}{\partial e_1} \mathfrak b_j\geq 0$.
\end{lemma}

\begin{proof}
Consider the expression for $\mathfrak b_j$ given in \eqref{eqn:betaj}, and observe that the denominator is not a function of $e_1$.  Hence it is sufficient to show that $\frac{\partial}{\partial e_1} \left(\prod_{i=1, i\ne j}^s T_i\right)\geq 0$ when $\mathbf{e}\in \mathbb R^s_{\geq 0}$, and this is immediately verified.
\end{proof}

\begin{lemma}\label{lem:jtozero}
Let $\mathbf{e}\in \mathbb R^s_{\geq 0}$ and fix some $j,k\in \{1, \dots, s\}$.
\begin{enumerate}
	\item\label{lem:jtozero:1} If $k<j$ then $\displaystyle \left(\frac{\partial}{\partial e_j}-\frac{\partial}{\partial e_k}\right)\mathfrak b_j\leq 0$.
	\item\label{lem:jtozero:2} If $k>j+1$ then $\displaystyle \left(\frac{\partial}{\partial e_{j+1}}-\frac{\partial}{\partial e_k}\right)\mathfrak b_j\leq 0$.
\end{enumerate}
\end{lemma}
\begin{proof}
Throughout this proof, we restrict all functions to the domain $\mathbf{e}\in \mathbb R^s_{\geq 0}$.  It is sufficient to prove the statements for $\log \mathfrak b_j$.   We may write:
\begin{equation}\label{eq:logbj}
\log \mathfrak b_j=\sum_{\overset{1\leq i \leq s}{i\ne j}} \log T_i -\sum_{i=1}^{j-1} \log U_{i,j}-\sum_{i=j+1}^s \log V_{i,j}.
\end{equation}
To prove part \eqref{lem:jtozero:1} of the lemma, we assume that $k<j$ and we fix some $i\in \{1, \dots, s\}$ where $i\ne j$.  Observe that
\[
\left(\frac{\partial}{\partial e_j}-\frac{\partial}{\partial e_k}\right) \log T_i=\left(\frac{\partial}{\partial e_j}-\frac{\partial}{\partial e_k}\right)\log (i+e_1+e_2+\dots +e_i)\leq 0
\]
with equality if and only if $i< k$ or $i> j$.  Similarly, if $i\in\{1, \dots, j-1\},$ then
\[
\left(\frac{\partial}{\partial e_j}-\frac{\partial}{\partial e_k}\right) \log U_{i,j}=\left(\frac{\partial}{\partial e_j}-\frac{\partial}{\partial e_k}\right)\log (j-i+1+e_{i}+e_{i+1}+\dots +e_{j}) \geq 0
\]
with equality if and only if $k<i$.  Since $k<j,$ we also have that
\[
\left(\frac{\partial}{\partial e_j}-\frac{\partial}{\partial e_k}\right) \log V_{i,j}=\left(\frac{\partial}{\partial e_j}-\frac{\partial}{\partial e_k}\right)\log( i-j+e_{j+1}+\dots +e_{i})=0
\]
for all $i\in\{j+1, \dots, s\}$.  By combining equation \eqref{eq:logbj} with the results of these three computations, we conclude that $(\frac{\partial}{\partial e_j}-\frac{\partial}{\partial e_k})(\log \mathfrak b_j)\leq 0$ as desired.

To prove part \eqref{lem:jtozero:2} of the lemma, we now assume that $k>j+1$ and we fix some $i\in \{1, \dots, s\}$ with $i\ne j$. Observe first that $\left(\frac{\partial}{\partial e_{j+1}}-\frac{\partial}{\partial e_k}\right)\log U_{i,j}=0$ for all $i$; second, that if $i<j+1$ or $i\geq k$ then $\left(\frac{\partial}{\partial e_{j+1}}-\frac{\partial}{\partial e_k}\right) \log T_i= 0$; and third, that if $i\geq k$ then $\left(\frac{\partial}{\partial e_{j+1}}-\frac{\partial}{\partial e_k}\right)\log V_{i,j}= 0$.  It remains to show that
\begin{equation*}\label{eq:balance}
\left(\frac{\partial}{\partial e_{j+1}}-\frac{\partial}{\partial e_k}\right) \sum_{i=j+1}^{k-1} \log T_i-\log V_{i,j}\leq 0.
\end{equation*}
This follows from the computation:

\begin{align*}
\left(\frac{\partial}{\partial e_{j+1}}-\frac{\partial}{\partial e_k}\right) \sum_{i=j+1}^{k-1} \log T_i-\log V_{i,j}&=\sum_{i=j+1}^{k-1} \frac{1}{i+e_1+\dots+e_i}-\frac{1}{i-j+e_{j+1}+\dots+e_i}\\
&=\frac{\left( i-j+e_{j+1}+\dots+e_i \right) - \left( i+e_1+\dots+e_i\right)}{\left(i+e_1+\dots+e_i\right)\left(i-j+e_{j+1}+\dots+e_i\right)}\\
&=\frac{-j-e_1-\dots-e_j}{\left(i+e_1+\dots+e_i\right)\left(j-i+e_{j+1}+\dots+e_i\right)}<0
\end{align*}
which completes the proof.
\end{proof}

\begin{lemma}\label{lem:pure}
Let $\mathbf{e}\in \mathbb R^s_{\geq 0}$ with $e_1\geq \sum_{j=1}^s e_j$.  Then
\[
\mathfrak b_j(\mathbf{e})\geq \binom{s}{j}.
\]
\end{lemma}
\begin{proof}
By Lemma~\ref{lem:upshifting}, it is sufficient to prove the lemma in the case $e_1=\sum_{i=2}^n e_i$.  Furthermore, by Lemma~\ref{lem:jtozero}, we may assume that $e_i=0$ for $i\notin \{1,j,j+1\}$.

Assume for the moment that $j\notin \{1,s\}$ and let $\widetilde{\mathbf{e}}=(e_1,e_j,e_{j+1})$.  Under these assumptions we may write $\mathfrak b_j$ as a function of $\widetilde{\mathbf{e}}$.  Our goal is to show that $\mathfrak b_j(\widetilde{\mathbf{e}})\geq \binom{s}{j}$ given the constraint $e_1=e_j+e_{j+1}$ and the domain $\widetilde{\mathbf{e}}\in \mathbb R^3_{\geq 0}$.

We introduce a new variable $t$ and write $e_j=te_1$ and $e_{j+1}=(1-t)e_1$.  Under this change of coordinates, our constrained minimization problem is now equivalent to minimizing the function:
\[
\mathfrak c_j:= \frac{(1+e_1)(2+e_1)\cdots (j-1+e_1)\cdot (j+1+2e_1)\cdots (n+2e_1)}{(j-1+t e_1)\cdots (1+t e_1) (1+(1-t)e_{1})\cdots ((n-j)+(1-t)e_1)}
\]
over the domain $(t, e_1)\in [0,1]\times [0,\infty)$.

We claim that the minimum of $\log \mathfrak c_j$ on the domain $[0,1]\times [0,\infty)$ occurs when $e_1=0$. The partial derivative $\frac{\partial \log \mathfrak c_j}{\partial e_1}$ is the sum of the following $4$ functions:
\begin{itemize}
    \item\label{dFe1}  $\displaystyle f_1:=\frac{1}{1+e_1}+\dots +\frac{1}{j-1+e_1}$
    \item\label{dFe2}  $\displaystyle f_2:=\frac{2}{j+1+2e_1}+\dots +\frac{2}{n+2e_1}$
    \item\label{dFe3}  $\displaystyle f_3:=-\frac{t}{1+t e_1} - \dots -\frac{t}{j-1+t e_1}$
    \item\label{dFe4}  $\displaystyle f_4:=-\frac{1 - t}{1+(1-t)e_1} - \dots -\frac{1-t}{(n-j)+(1-t)e_1}$
\end{itemize}
We observe first that:
\begin{align*}
-f_1-f_3&=\sum_{i=1}^{j-1} \frac{-1}{i+e_1}+\frac{t}{i+te_1}\\
&=\sum_{i=1}^{j-1} \frac{-(i+te_1)+t(i+e_1)}{(i+e_1)(i+te_1)}\\
&=\sum_{i=1}^{j-1} \frac{(-1+t)i}{(i+e_1)(i+te_1)}.
\end{align*}
Hence $-f_1-f_3\leq 0$ whenever $(t,e_1)\in [0,1]\times [0,\infty)$.  We next observe that:
\begin{align*}
-f_2-f_4&=\sum_{i=j+1}^n \frac{-2}{i+2e_1}+\frac{1-t}{i+(1-t)e_1}\\
&=\sum_{i=j+1}^n \frac{\left(-2i-2(1-t)e_1\right)+\left( (1-t)i+2(1-t)e_1\right)}{(i+2e_1)(i+(1-t)e_1)}\\
&=\sum_{i=j+1}^n \frac{-i-it}{(i+2e_1)(i+(1-t)e_1)}.
\end{align*}
Hence $-f_2-f_4\leq 0$ whenever $(t,e_1)\in [0,1]\times [0,\infty)$.  Combining these two observations, we have that:
\[
-\frac{\partial \log \mathfrak c_j}{\partial e_1}=-f_1-f_2-f_3-f_4\leq 0
\]
on the domain $[0,1]\times [0,\infty)$.  A minimum of the function $\log \mathfrak c_j$ on the domain $[0,1]\times [0,\infty)$ thus occurs when $e_1=0$, and it follows that the same statement holds for the function $\mathfrak c_j$.  Direct computation yields that $\mathfrak c_j(t,0)=\binom{s}{j}$, which completes the proof when $j\notin \{1,s\}$.

If $j=1$, then we may still apply Lemma~\ref{lem:jtozero} and reduce to the case that $e_i=0$ for $i\ne 1$.  Then we have:
\[
\mathfrak b_1(e_1)=\frac{(2+e_1)(3+e_1)\cdots (s+e_1)}{(s-1)!}
\]
which is at least than $\binom{s}{1}$ whenever $e_1\geq 0$.  If $j=s$, we reduce to the case that $e_s=e_1$ and we have:
\[
\mathfrak b_s(e_1,e_s)=\frac{(1+e_1)\cdots (s-1+e_1)}{(s-1)!}
\]
which is at least $\binom{s}{s}$ whenever $e_1\geq 0$.
\end{proof}

\begin{cor}\label{cor:pures}
Let $\mathbf{d}\in \mathbb Z^{s+1}$ such that $d_0\leq 0$ and such that $d_s-s\leq 2d_1-2$.  Then:
\[
\beta_j(\overline{\pi}(\mathbf{d}))\geq \binom{s}{j}
\]
\end{cor}
\begin{proof}
Let $\mathbf{e}=(e_1, \dots, e_s)$ where $e_i=d_i-d_{i-1}-1$, so that $\mathfrak d(\mathbf{e})=\mathbf{d}$.  Since $d_s=d_0+s+\sum_{i=1}^s e_i$ and $d_1=d_0+e_1+1$, we have that:
\[
\sum_{i=2}^s e_i=\left(d_s-s\right)-d_0-e_1\leq \left(2d_1-2\right)-d_0-e_1=d_1-1\leq e_1.
\]
The corollary now follows from Lemmas~\ref{lem:defs} and \ref{lem:pure}.
\end{proof}

%%%%%%%%%%
%%%%%%%%%%
\section{Proof of Theorem \ref{thm:weak}}\label{sec:proof}
%%%%%%%%%%
%%%%%%%%%%
We now prove our main result.
\begin{proof}[Proof of Theorem \ref{thm:weak}]
By Theorem~\ref{thm:boijsod}, we may write the Betti diagram of $M$ as a positive rational sum of pure diagrams:
\begin{equation}\label{eq:boijsoddecomp}
\beta(M)=\sum_{c\leq s \leq p} \sum_{\overset{d\in \deg(\mathbb Z^{p+1})}{\tau_s(\underline{\mathbf{d}})\leq \mathbf{d} \leq \tau_s(\overline{\mathbf{d}})}}a_{\mathbf{d}} \overline{\pi}(\mathbf{d})
\end{equation}
By linearity, it is sufficient to show that:
\begin{equation}\label{eq:betti}
\beta_j(\overline{\pi}(\mathbf{d}))\geq \binom{c}{j}
\end{equation}
for every pure diagram appearing with nonzero coefficient in \eqref{eq:boijsoddecomp} and for every $j\in \{1, \dots, c\}$.  Let $\mathbf{d}=(d_0, \dots, d_s)$ be a degree sequence corresponding to such a pure diagram, and let $\mathbf{e}=(e_1, \dots, e_s)$ defined by $e_i:=d_i-d_{i-1}-1$.  Since $\overline{\pi}(\mathbf{d})$ appears with positive coefficient in equation~\eqref{eq:boijsoddecomp}, it must contribute to the Betti diagram $\beta(M)$.  It follows that $d_0\leq 0$ and that
\[
d_s-s\leq \reg(M)\leq 2\underline{d}_1(M)-2\leq 2d_1-2.
\]
Hence $\mathbf{d}$ satisfies the hypotheses of Corollary~\ref{cor:pures}, and $\beta_j(\overline{\pi}(\mathbf{d}))\geq \binom{s}{j}$.  Since $s\geq c$, it follows that $\binom{s}{j}\geq \binom{c}{j}$, and we obtain inequality~\eqref{eq:betti}.
\end{proof}

\begin{remark}
With more care, one could show that equality in Theorem~\ref{thm:weak} may only occur in cases where $\codim(M)\leq 2$ or where there exists $m\in \mathbb N$ such that $M\cong k^m$ as a graded $S$-module.
\end{remark}

%%%%%%%%%%%
%%%%%%%%%%%
\section{Asymptotic Betti Numbers}\label{sec:asymp}
%%%%%%%%%%%
%%%%%%%%%%%
Let $I$ be an ideal generated in a single degree $\delta$.  The regularity of $I^t$ becomes a linear function $\reg(I^t)=\delta t+b$ for $t\gg 0$ (c.f. \cite[Thm 3.2]{trung-wang}, \cite[Cor 3]{kodiyalam}).
We define $b$ to be the \defi{asymptotic regularity defect} of $I$.  The following theorem gives lower bounds for the Betti numbers of $S/I^t$.
\begin{thm}\label{thm:asymp}
Let $I$ be an ideal of codimension $c$ generated in a single degree $\delta$ and with asymptotic regularity defect $b$.
We have the following lower bound on the Betti numbers of $S/I^t$:
\[
\beta_j(S/I^t) \geq \frac{(b!)^2\delta^{c-1}}{(j-1+b)!(c-j+b)!}t^{c-1} + O(t^{c-2})
 \]
for all $j=1, \dots, c$ and for all $t\gg0$.
\end{thm}

\begin{proof}
By Theorem~\ref{thm:boijsod} we may write $\beta(S/I^t)$ as a sum of pure diagrams as in equation~\eqref{eq:boijsoddecomp}.  Let $\mathbf{d}=(d_0, \dots, d_s)$ be some degree sequence such that $\overline{\pi}(\mathbf{d})$ appears with nonzero coefficient in this sum.  The equality $\codim(I^t)=\codim(I)=c,$ implies that $s\geq c$.  Let $\mathbf{e}=(e_1, \dots, e_s)$ defined by $e_i=d_i-d_{i-1}-1$.  Since $I^t$ is generated in degree $t\delta$, we have $e_1=t\delta$.  Let $t\gg0$ so that $\reg(S/I^t)=t\delta+b$.  Since $\reg(S/I^t)=\overline{d}_p(S/I^t)-p$ we have that $\sum_{i=2}^s e_i\leq b$.

It is sufficient to prove the lower bound for the Betti numbers of the pure diagram $\overline{\pi}(\mathbf{d})$.  In fact, it is sufficient to prove the lower bound for the functions $\mathfrak b_j(\mathbf{e})$ under the constraints $e_1=t\delta$ and $\sum_{i=2}^s e_i\leq b$.  Let $j\in \{1, \dots, c\}$.  By Lemma~\ref{lem:jtozero}, we may assume that $e_i=0$ unless $i\in \{1,j,j+1\}$.  Hence we reduce to the case that $e_j+e_{j+1}\leq b$.  We now seek to compute $\mathfrak b_j$.
\[
\mathfrak b_j(\mathbf{e})= \frac{(1+t\delta)(2+t\delta)\cdots (j-1+t\delta)(j+1+t\delta+b) \cdots(s+t\delta+b)}{(j-1+e_j)\cdots (1+e_j)(1+e_{j+1})\cdots (s-j+e_{j+1})}
\]
Note that $e_j$ and $e_{j+1}$ only appear in the denominator, and both are positive numbers less than $b$.  Hence setting $e_j=e_{j+1}=b$ only decreases the right-hand side.  This yields:
\begin{equation}\label{eq:asymp1}
\mathfrak b_j(\mathbf{e})\geq \frac{(1+t\delta)(2+t\delta)\cdots (j-1+t\delta)(j+1+t\delta+b) \cdots(s+t\delta+b)}{(j-1+b)\cdots (1+b)(1+b)\cdots (s-j+b)}
\end{equation}
Since $s\geq c$ we may rewrite the right-hand side of \eqref{eq:asymp1} as
\[
\left( \frac{(1+t\delta)(2+t\delta)\cdots (j-1+t\delta)(j+1+t\delta+b)\cdots(c+t\delta+b)}{(1+b)\cdots (j-1+b)(1+b)\cdots (c-j+b)} \right)\left( \prod_{i=1}^{s-c} \frac{(c+i+t\delta+b)}{(c+i-j+b)}
\right).
\]
Each term of the product on the right is greater than $1$, so by deleting this product and substituting back into \eqref{eq:asymp1}, we obtain the inequality:
\begin{align*}
\mathfrak b_j(\mathbf{e})&\geq \frac{(1+t\delta)(2+t\delta)\cdots (j-1+t\delta)(j+1+t\delta+b)\cdots(c+t\delta+b)}{(1+b)\cdots (j-1+b)(1+b)\cdots (c-j+b)}\\
&=\frac{(b!)^2\delta^{c-1}}{(j-1+b)!(c-j+b)!}t^{c-1} + O(t^{c-2}).
\end{align*}
This completes the proof.
\end{proof}

%%%%%%%%%%%
%%%%%%%%%%%
\section{Examples}\label{sec:apps}
%%%%%%%%%%%
%%%%%%%%%%%
In this section, we consider several applications of Theorem~\ref{thm:weak}, and we remark on the necessity of the hypothesis that $\reg(M)\leq 2\underline{d}_1(M)-2$.

\begin{example}
Let $V\subseteq S_d$ be any vector space of forms of degree $d$ with $d>1$, and let $I\subseteq S$ be the ideal $V+\mathfrak m^{d+1}$.  Then $S/I$ satisfies the hypotheses of Theorem~\ref{thm:weak}.  More generally, if $e\leq 2d-1$ and $J$ is the ideal generated by $V$ and by $\mathfrak m^e$, then $S/J$ satisfies the hypotheses of Theorem~\ref{thm:weak}.
\end{example}

\begin{example}[Curves of High Degree]
Let $C\subseteq \mathbb P^{n}$ be a smooth curve of genus embedded by a complete linear system of degree at least $2g+1$.  Let $I_C\subseteq k[x_0, \dots, x_n]$ be the ideal defining $C$.  Then $\reg(S/I_C)\leq 2$ by \cite[Corollary 8.2]{eis-syzygy}. Hence $S/I_C$ satisfies the hypotheses of Theorem~\ref{thm:weak}, and thus $\beta_i(S/I_C)\geq \binom{n-1}{i}$.
\end{example}

\begin{example}[Toric Surfaces]\label{ex:toric}
Let $X\subseteq \mathbb P^n$ be a toric surface embedded by a complete linear system $|A|$.  Let $I_X\subseteq S=k[x_0, \dots, x_n]$ be the defining ideal of $X$.  We claim that $S/I_X$ satisfies the hypotheses of Theorem~\ref{thm:weak}, and hence that $\beta_i(S/I_X)\geq \binom{n-2}{i}$.  Since $I_X$ has no generators in degree $1$, this amounts to showing that $\reg(S/I_X)\leq 2$.  It is equivalent to show that the sheaf $\mathcal I_X:=\widetilde{I_X}$ is $3$-regular~\cite[Prop 4.16]{eis-syzygy}.

We first check that $H^1(\mathbb P^n,\mathcal I_X(2))=0$.  Since $X$ is a toric surface and $A$ is an ample divisor, it follows from, for instance~\cite[Cor 2.1]{schenck}, that $X$ satisfies condition $N_0$, and hence that $X$ is projectively normal.  The surjectivity of the map
\[
H^0(\mathbb P^n, \mathcal O_{\mathbb P^n}(2))\to H^0(X, \mathcal O_X(2))
\]
and the vanishing of $H^1(\mathbb P^n, \mathcal O_{\mathbb P^n}(2))$ then imply that $H^1(\mathbb P^n,\mathcal I_X(2))=0$.

We next check that $H^2(\mathbb P^n, \mathcal I_X(1))=0$.  This follows from the exact sequence:
\[
H^1(X,\mathcal O_X(1))\to H^2(\mathbb P^n, \mathcal I_X(1))\to H^2(\mathbb P^n, \mathcal O_{\mathbb P^n}(1))
\]
and the fact that higher cohomology of ample line bundles vanishes on toric varieties.  We conclude that $\mathcal I_X$ is $3$-regular, which implies that $S/I_X$ satisfies the hypotheses of Theorem~\ref{thm:weak}.
\end{example}

\begin{example}
Let $I$ be any ideal with minimal degree generator in degree $\underline{d}_1$ and maximal degree generator in degree $\overline{d}_1$.  Assume that $\overline{d}_1(I)<2\underline{d}_1(I)$.  Then
\[
\reg(S/I^t)\leq t\overline{d}_1+b
\]
for some $b$ and for all $t\geq 1$~\cite[Thm 1.1(i)]{cht}.  Since $\overline{d}_1(I)<2\underline{d}_1(I)$, it follows that, for all $t\gg0$, $t\overline{d}_1(I)+b<2t\underline{d}_1(I)-2$.  Hence, for all $t\gg0$, the module $S/I^t$ satisfies the hyoptheses of Theorem~\ref{thm:weak}.
\end{example}

The method of proof for Theorem~\ref{thm:weak} breaks down if one removes the hypothesis that $\reg(M)\leq 2\underline{d}_1(M)-2$.  One issue is that the statement:
\[
\beta_j(M)\geq \beta_0(M) \binom{\codim(M)}{j}
\]
is {\em not} true in general.  For example, if $S=k[x,y]$ and $N$ is the cokernel of a generic $2\times 3$ matrix of linear forms, then
\[
\beta_1(N)=3< 4=\beta_0(N) \binom{\codim(N)}{1}.
\]
There also exist pure diagrams with integral entries which do not satisfy the graded BEH rank conjecture.  For instance, the diagram:
\[
\overline{\pi}(0,1,2,3,5,6)=\begin{pmatrix}
1&\frac{9}{2}&\frac{15}{2}&5&-&-\\
-&-&-&-&\frac{3}{2}&\frac{1}{2}
\end{pmatrix}
\]
does not satisfy any version of Conjecture~\ref{conj:beh}.

%%%%%%%%%%%%%
%%%%%%%%%%%%%
\section*{Acknowledgements}
We thank Kiril Datchev for helpful discussions about the proof of Lemma~\ref{lem:pure}, David Eisenbud for his guidance, Milena Hering for suggesting improvements to Example~\ref{ex:toric}, and Tony V\'{a}rilly-Alvarado for comments on an early draft.

%\bibliographystyle{alpha}
%\begin{thebibliography}{Maz80}

%%%%%%%%%%%%
%%%%%%%%%%%%
%%%%%%%%%%%%

\end{document}